\documentclass[12pt]{article}
\usepackage{amsmath,amsfonts,amssymb,amsthm}
\usepackage[all]{xy}

\newcommand{\A}{\mathbb A}

\newcommand{\N}{\mathbb N}
\newcommand{\Q}{\mathbb Q}
\newcommand{\Z}{\mathbb Z}

\newcommand{\cF}{\mathcal F}
\newcommand{\cG}{\mathcal G}
\newcommand{\cH}{\mathcal H}

\newcommand{\cI}{\mathcal I}
\newcommand{\PP}{\mathbb P}
\newcommand{\cO}{\mathcal O}
\newcommand{\cP}{\mathcal P}
\newcommand{\omicron}{o}
\newcommand{\cL}{\mathcal L}

\newcommand{\Ga}{\mathbb{G}_a}
\newcommand{\Gm}{\mathbb{G}_m}

\newcommand{\magma}{{\sc Magma}}
\newcommand{\broken}{\dashrightarrow}
\renewcommand{\P}{\mathbb{P}}
\renewcommand{\phi}{\varphi}
\DeclareMathOperator{\rank}{rank}

\DeclareMathOperator{\Pic}{Pic}

\DeclareMathOperator{\ch}{char}

\DeclareMathOperator{\mult}{mult}
\DeclareMathOperator{\Nonsing}{Nonsing}

\DeclareMathOperator{\Exc}{Exc}
\DeclareMathOperator{\Bs}{Bs}
\DeclareMathOperator{\red}{red}
\DeclareMathOperator{\rad}{Rad}
\DeclareMathOperator{\Bir}{Bir}
\DeclareMathOperator{\Div}{Div}
\DeclareMathOperator{\Hom}{Hom}
\DeclareMathOperator{\Supp}{Supp}
\newcommand{\tefrac}{\textstyle\frac}
\newcommand{\tesum}{\textstyle\sum}
\renewcommand{\ge}{\geqslant}
\renewcommand{\le}{\leqslant}
\newcommand{\kbar}{\overline{k}}
\newcommand{\Eamb}{E_{\mathrm{amb}}}
\newcommand{\Xtil}{\widetilde{X}}

\swapnumbers
\theoremstyle{plain}
\newtheorem{thm}{Theorem}[section]

\newtheorem*{prop*}{Proposition}

\theoremstyle{definition}

\newtheorem{definition}[thm]{Definition}

\newtheorem{rk}[thm]{Remark}
\newtheorem*{definition*}{Definition}

\newcommand{\XmuH}{(X, \tefrac{1}{\mu} \cH)}
\newcommand{\KXmuH}{K_X + \tefrac{1}{\mu} \cH}
\newcommand{\qeq}{\sim_{\Q}}

\title{Elliptic fibrations on cubic surfaces}
\author{Gavin Brown \and Daniel Ryder}
\date{19 June 2008}

\begin{document}
\maketitle

\begin{abstract}
We classify elliptic fibrations birational to a nonsingular,
minimal cubic surface over a field of characteristic zero.
Our proof is adapted to provide computational techniques for
the analysis of such fibrations, and we describe an implementation
of this analysis in computer algebra.
\end{abstract}

\setcounter{tocdepth}{2}

\section{Introduction}

Let $X\subset\P^3$ be a nonsingular projective cubic surface
over a field~$k$ of characteristic zero.  An \emph{elliptic
fibration on~$X$}, sometimes called an \emph{elliptic fibration
birational to~$X$}, is a dominant rational map $\phi\colon X\broken B$
to a normal variety~$B$, where $\phi$ is defined over~$k$, it has
connected fibres, and its general geometric fibre
is birational to a curve of genus~1.

We describe in Section~\ref{sec!hal} a class of elliptic fibrations
called {\em Halphen fibrations}.  Conversely, given an elliptic
fibration on a \emph{minimal\/}~$X$ (see below)
we relate it to an Halphen fibration as follows:
\begin{thm}
\label{thm!main}
Let $X \subset \PP^3$ be a nonsingular, minimal
cubic surface over a field~$k$ of characteristic zero.
If $\phi \colon X \broken B$
is an elliptic fibration on~$X$ then $B\cong\PP^1$
and there exists a composite
\[ \xymatrix{ X \ar@{-->}[r]^{i_s} & X
\ar@{-->}[r]^{i_{s-1}} & \quad \cdots \quad
\ar@{-->}[r]^{i_2} & X \ar@{-->}[r]^{i_1} & X } \]
of birational selfmaps of~$X$, each of which is a Geiser
or Bertini involution, such that $\;\phi \circ i_1 \circ
\cdots \circ i_s \colon X \broken B\cong\PP^1\;$ is an Halphen fibration.
\end{thm}
Geiser and Bertini involutions are birational selfmaps
of~$X$ described in Section~\ref{sec!GB}.
This result is proved in Cheltsov \cite{Ch} and independently
in the unpublished~\cite{R00}.  Our aims and methods are
different from those of~\cite{Ch}, however: we seek to be as
explicit as possible, and we have implemented algorithms in the
computational algebra system
\magma~\cite{Ma} for Halphen fibrations
and Geiser and Bertini involutions.
Our code is available at~\cite{BR}.

All varieties, subschemes, maps and linear
systems are defined over the fixed field~$k$ of characteristic zero,
except where a different field is mentioned explicitly.

\paragraph{Contents of the paper.}
In the remainder of the introduction we discuss motivation and background
for the problem.  We build Halphen fibrations
on~$X$ in Section~\ref{sec!con}. In Section~\ref{sec!proof} we
discuss the Noether--Fano--Iskovskikh inequalities and then
prove Theorem~\ref{thm!main}.  Section~\ref{sec!alg} is devoted
to algorithmic considerations and an outline of our implementation,
while Section~\ref{sec!egs} contains worked computer examples.

\paragraph{Cubic surfaces and minimality.}
Throughout this paper, by {\em cubic surface\/} we mean a
nonsingular surface $X\subset\PP^3$ defined by a homogeneous
polynomial of degree~3 with coefficients in~$k$.
We denote $-K_X$ by~$A$.

When $k$~is algebraically closed, it is well known that $X$
contains 27~straight lines
and that these span the Picard group $\Pic(X)\cong\Z^7$.
One quickly deduces that there is a birational map
$X\broken\P^2$; in other words, $X$ is~{\em rational}.  On the
other hand, if $k$ is not algebraically closed, some of these
lines may fail to be defined over~$k$ and the Picard group may
have smaller rank.  Indeed, $\Pic(X)$ is the Galois-invariant
part of~$\Pic(\overline{X})$.
A cubic surface $X$ is {\em minimal\/} if the Picard number of~$X$,
$\rho(X)=\rank\Pic(X)$, is~1.
It is easy to see that if $X$ is minimal then $\Pic(X) = \Z({-K_X})$.

Elliptic fibrations were defined above with apparently arbitrary
base~$B$, but in fact it follows from Iitaka's bound on Kodaira dimension
that $g(B)=0$ for any surface $X$ of Kodaira dimension~$-\infty$;
see \cite{BHPV} Theorem~(18.4). In particular this applies to cubic
surfaces and so we have:
\begin{prop*}
If $X\broken B$ is an elliptic fibration on a cubic surface~$X$
then $g(B)=0$.
\end{prop*}
\noindent
There remains the question of whether $B$ has a $k$-rational
point, that is, whether $B\cong\PP^1$ over~$k$; we return
to this in Section~\ref{sec!B}.

\paragraph{Geometric motivation.}
Our main motivation for studying elliptic fibrations on cubic
surfaces is geometric. This is best explained from a broader perspective.

A {\em Fano $n$-fold\/} is a
normal projective variety~$X$ of dimension~$n$, with at worst
$\Q$-factorial terminal singularities and Picard number~1, such
that ${-K_X}$ is ample.
A fundamental question in Mori theory is whether a given Fano
$n$-fold~$X$ admits  birational maps to other Mori fibre
spaces --- see~\cite{Co} for a discussion, noting that a key example
of such  {\em birational non-rigidity\/} is a rational
map $\phi\colon X\broken S$ whose generic fibre is a curve of
genus~0 rather than~1.  We regard the search for elliptic fibrations
as a limiting case in Mori theory --- a
point of view we learned from papers of 
Iskovskikh~\cite{Is} and Cheltsov~\cite{Ch},
and one that becomes clearer
when we discuss the Noether--Fano--Iskovskikh inequalities in
Section~\ref{sec!prelim}.  For more on how our problem fits into
modern birational geometry, see~\cite{Is} and the
introduction to~\cite{CPR}.

\paragraph{Arithmetic motivation.}
Cases of the more general
problem of classifying elliptic fibrations on Fano varieties also
have arithmetic applications.  From this point of view a
cubic surface is a baby case; but scaled-up versions
of our methods attack, for example, the same problem for some Fano
3-folds, see \cite{Ch} and~\cite{R06}.

In arithmetic a basic question concerning Fano varieties is the
existence, or at least potential density, of rational points.
Elliptic fibrations offer one approach; see Bogomolov and
Tschinkel~\cite{BT}, for instance.

\paragraph{History.}
In contrast to the modern motivation, some of the methods are ancient.
In his paper~\cite{H} of~1882 Halphen considered the problem of
finding a plane curve~$G$ of degree~6 with 9~prescribed double
points $P_1,\ldots,P_9$. The question is: for which collections of
points~$\{P_i\}$ is there a solution
apart from $G = 2C$, where $C$~is the (in general unique)
cubic containing all the~$P_i$?  Halphen's answer is that $C$
must indeed be unique and --- in modern language and supposing
for simplicity that $C$ is nonsingular, so
elliptic --- $P_1\oplus\cdots\oplus P_9$ must be a nonzero
2-torsion point of~$C$, where any inflection point is chosen as
the zero for the group law.  He proceeds to consider higher
torsion as well.  Translated to a cubic surface, this is
essentially Theorem~\ref{thm!constr}.  A natural next step is the
result analogous to Theorem~\ref{thm!main} for~$X=\PP^2$, and
this was proved by Dolgachev~\cite{D} in~1966.

The approach of~\cite{Ch} to Theorem~\ref{thm!main}
is considerably more highbrow than
ours: he uses general properties of mobile
log pairs and does not spell out the construction of elliptic fibrations
in detail. The paper~\cite{R00}, on the other hand, was originally
conceived as a test case for~\cite{R02}
and~\cite{R06}, which concern similar problems for Fano 3-folds.

\paragraph{Acknowledgments.}
It is our pleasure to thank Professors Andrew Kresch and Miles Reid
for their help with some finer points of arithmetic and Professor Josef Schicho
for a preview of his new \magma\ package to compute the Picard group
of a cubic surface over a non-closed field.

\section{Constructing elliptic fibrations}
\label{sec!con}

We fix a nonsingular, minimal cubic surface~$X$ defined over~$k$,
with $A=-K_X$.  Linear equivalence of divisors is denoted
by~$\sim$ and $\Q$-linear equivalence by~$\sim_{\Q}$.

\subsection{Halphen fibrations}
\label{sec!hal}

The simplest elliptic fibrations arise as the pencil
of planes through a given line. That is, if $L=(f=g=0)$
is a line in~$\P^3$ defined by two independent linear
forms~$f,g$ and not lying wholly in~$X$, then the map
$\phi = (f,g)$ is an elliptic fibration.
In this section we make a larger class of fibrations which
includes these linear fibrations as a simple case.
\begin{definition}
\label{def!hal}
A pair $(G,D)$ is called {\em Halphen data on~$X$} when 
$G\in|A|$ is (reduced and) irreducible over~$k$ and 
$D\in\Div(G)$ is an effective $k$-rational divisor of degree~3,
supported in the nonsingular locus of~$G$, satisfying
$\cO_G(\mu D)\cong \cO_G(\mu A)$ for some integer $\mu\ge 1$.
The smallest such $\mu\ge1$ is called the \emph{index} of~$(G,D)$.
\end{definition}

Since $X$ is minimal, $G$ may be any irreducible plane cubic or the
union of three conjugate lines (it is required to be irreducible
over~$k$, not over~$\kbar$).  Since $\Supp(D) \subset
\Nonsing(G)$, the sheaf isomorphism condition says that
$A_{|G}-D$ is a torsion class of order~$\mu$ in~$\Pic(G)$.

\begin{definition} \label{def!res}
Let $(G,D)$ be Halphen data on~$X$.
The {\em resolution of~$(G,D)$\/} is the blowup $\pi\colon
Y\rightarrow X$ of a set of up to three points~$P_i$ that lie on
varieties dominating~$X$ and are determined as follows:
\begin{itemize}
\item[A1.]  If $D$ is a sum of distinct $k$-rational points of~$G$
then let $\{P_1,P_2,P_3\}=\Supp(D)$ (as points of~$X$)
and let $\pi$~be the blowup of these points.
\item[A2.]  If $D = p + 2q$, where $p \ne q$ are $k$-rational
points of~$G$, then let $P_1 = p$ and $P_2 = q$ (as points
of~$X$); also let $\xi \colon Y' \to X$ be the blowup of these
points and let $E_2'$ be the exceptional curve lying over~$P_2$.
Now define $P_3$ to be the point $G' \cap E_2'$ on~$Y'$, where
$G'$ is the strict transform of~$G$; let $\omicron \colon Y \to
Y'$ be the blowup of~$P_3$ and set $\pi = \xi \circ \omicron$.
\item[A3.]  If $D = 3p$ with $p$ a $k$-rational point of~$G$ then
let $P_1 = p$ and let $\nu \colon Y' \to X$ be the blowup
of~$P_1$.  Next define $P_2 = E_1' \cap G'$ where $E_1',G'
\subset Y'$ are respectively the exceptional curve of~$\nu$ and
the strict transform of~$G$, and let $\xi \colon Y'' \to Y'$ be
the blowup of~$P_2$.  Now, similarly, define $P_3 = E_2'' \cap
G''$ where $E_2'',G'' \subset Y''$ are respectively the
exceptional curve of~$\xi$ and the strict transform of~$G'$.
Finally let $\omicron \colon Y \to Y''$ be the blowup of~$P_3$ and
let $\pi = \nu \circ \xi \circ \omicron \colon Y \to X$.
\item[B.]  If $D = p_1 + p_2$ with $p_1$ a $k$-rational
point of~$G$ and $\deg(p_2) = 2$ then let $P_i = p_i$ for $i =
1,\,2$ and let $\pi \colon Y \to X$ be  the blowup of~$P_1$
and~$P_2$.
\item[C.]  If $D = p$, a single $k$-closed point of~$G$ of
degree~3, then let $P_1 = p$ and let $\pi\colon Y\rightarrow X$
be the blowup of~$P_1$.
\end{itemize}
In each case we fix the following notation: let $E_i \subset Y$
be the \emph{total\/} transform on~$Y$ of the exceptional curve
over~$P_i$.  So in case~A2, for example, $E_2 = \omicron^*(E_2')
= E_2'' + E_3$ has two irreducible components, $E_2'' =
\omicron^{-1}_*(E_2')$ and $E_3 = \Exc(\omicron)$.  Furthermore
let $E = \sum_i E_i$, the relative canonical class of~$\pi$.
\end{definition}

It can easily be checked in the above definition that $E_i$ is
the reduced preimage of~$P_i$ on~$Y$.  Note, though, that this is
a consequence of our positioning of each subsequently-defined
$P_j$ on the strict transform of~$G$; the corresponding statement
no longer holds, for example, in the closely related
notation of Section~\ref{sec!prelim} below.

\begin{definition} \label{def!H}
Let $(G,D)$ be Halphen data on~$X$ of index~$\mu$, and let
$\pi\colon Y\rightarrow X$ be the resolution introduced above
with relative canonical class~$E$.
We define $\cH_Y$ to be the linear
system $|\mu \pi^*(A) - \mu E|$ on~$Y$.
The {\em Halphen system $\cH$ associated to $(G,D)$} is the
birational transform of~$\cH_Y$ on~$X$.
\end{definition}

Notice that $\cH$ is the set of divisors in~$|\mu A|$
that have multiplicity $\mu$ at every point~$P_i$.
It would be natural to write $\cH = |\mu A - \mu D|$, but
we don't.

\begin{thm} \label{thm!constr}
Let $(G,D)$ be Halphen data on~$X$ of index~$\mu$, and let
$\cH$ be the linear system described in Definition~\ref{def!H}.
Then $\cH$ is a mobile pencil, and the rational map $\phi =\phi_\cH$
is an elliptic fibration $\phi\colon X \broken \PP^1$
that has $\mu G$ as a fibre.
The set-theoretic base locus of~$\phi$ is~$\Supp(D)$ and
the resolution of~$(G,D)$ is its minimal resolution of indeterminacies.
\end{thm}

Following Cheltsov \cite{Ch}, fibrations $\phi_\cH$ arising in this
way are called {\em Halphen fibrations}. 
We give the proof of this theorem in Section~\ref{sec!pf}.

\subsection{Twisting by Geiser and Bertini involutions}
\label{sec!GB}

Not all elliptic fibrations are Halphen: we can precompose, or
{\em twist}, Halphen fibrations by elements of~$\Bir(X)$, and
usually the result will have more than three basepoints (counted
with degree).

We describe two particular classes of birational selfmap
of~$X$: Geiser and Bertini involutions, also described at greater
length in \cite{CPR} Section~2.
In fact, the group $\Bir(X)$ (in the case of minimal~$X$) is
generated by Geiser and Bertini involutions together with all
regular automorphisms, although we do not use this fact
explicitly; see \cite{M}~Chapter~5.

\paragraph{Geiser involutions.}
Let $P\in X$ be a point of degree~$1$.
We define a birational map $i_P\colon X\broken X$ as follows.
Let $Q$ be a general point of~$X$, and let $L\subset\P^3$ be the line
joining $P$ to~$Q$. Then $L\cap X$ consists of three distinct
points, $P,Q$ and a new point~$R$.  Define $i_P(Q)=R$.
In fact, $i_P$ is the map defined by the linear system~$|2A-3P|$.

\paragraph{Bertini involutions.}
Let $P\in X$ be a point of degree~$2$.
Let $L\subset\P^3$ be the unique line that contains~$P$.
Since $X$ is minimal, $L$ intersects $X$ in $P$ and exactly
one other point $R$ of degree~$1$.
We define a birational map $i_P\colon X\broken X$ as follows.
Let $Q$ be a general point of~$X$.
If $\Pi\cong\P^2$ is the plane spanned by $P$ and~$Q$, then
$C=\Pi\cap X$ is a nonsingular plane cubic curve containing~$R$.
Then $i_P(Q) = {-Q}$, the inverse of~$Q$ in the group law on
$C$ with origin~$R$.
In fact, $i_P$ is the map defined by the linear system~$|5A-6P|$.

\subsection{Proof of Theorem~\ref{thm!constr}}
\label{sec!pf}

\paragraph{Comments about $G$.}

We are given Halphen data $(G,D)$ on~$X$.
The curve $G$ is a Gorenstein scheme
with $\omega_G\cong\cO_G$ and~$\chi(\cO_G)=0$.  

When $\mu>1$, $G$ cannot be a cuspidal cubic since in that case
the Picard group $\Pic(G) \cong \Ga$ is torsion free; here we
use~$\ch(k)=0$. This restriction on~$G$ also follows from
Theorem~\ref{thm!constr}, given Kodaira's classification of
multiple fibres of elliptic fibrations: multiple cusps do not
occur.  Our $G$ may be a nodal cubic (with Picard group~$\Gm$) or
a triangle of conjugate lines (with Picard group an extension
of~$\Z^3$ by~$\Gm$).  If~$\mu = 1$ then $G$ can be cuspidal; but
in this case we are free to re-choose~$G$ as we please from the
pencil~$\cH$ of Definition~\ref{def!H}, so without loss of
generality $G$ is nonsingular.

\begin{proof}[Proof of Theorem~\ref{thm!constr}]
The case $\mu=1$ is trivial, so let $\mu \ge 2$.

Let $\pi\colon Y\rightarrow X$ together with the points~$P_i$ be
the resolution of~$(G,D)$ of Definition~\ref{def!res}.  We have
the Halphen system $\cH$ on~$X$ of Definition~\ref{def!H} and, by
construction,~$\mu G\in \cH$.

Suppose at first that we are in case A1, B or~C.
Define $\cF$ on~$X$ as the tensor product of all~$\cI_{P_i}^{\mu}$.
There is a map between exact sequences of sheaves of $\cO_X$-modules:
\begin{equation} \label{eq!Gdef}
\xymatrix@C=0.8cm{
0 \ar[r] &
\cF(\mu A) \ar[r] \ar[d] &
\cO_X(\mu A) \ar[r] \ar[d] &
\cG \ar[r] \ar[d] &
0
\\
0 \ar[r] &
\cO_G(\mu A-\mu D) \ar[r] &
\cO_G(\mu A) \ar[r] &
\cO_{\mu D}(\mu A) \ar[r] &
0
} \end{equation}
%
where $\cG=\left(\cO_X/\cF\right)\otimes\cO_X(\mu A)$.
(The lefthand vertical arrow is from the definition of~$\cF$, the
central one is clear, and the final one follows from the others.)
By assumption, $\cO_G(\mu A-\mu D)\cong\cO_G$.

Kodaira vanishing shows that $H^1(X,\cO_X(\mu A))=0$.
By Serre duality (since $G$ is Gorenstein) we have
\[
H^1(G,\cO_G(\mu A))
\cong
\Hom(\cO_G(\mu A),\cO_G)^*,
\]
and this $\Hom$ is zero because $A$ is ample on every
component of~$G$.
So, taking cohomology, we have a map between exact
sequences of $k$-vector spaces:
\[ \xymatrix@C=0.4cm{
0 \ar[r] &
H^0(X,\cF(\mu A)) \ar[r] \ar[d] &
H^0(X,\cO_X(\mu A)) \ar[r] \ar[d] &
H^0(X,\cG) \ar[r] \ar[d]^{\beta_1} &
H^1(X,\cF(\mu A)) \ar[r] \ar[d]^{\beta_2} &
0
\\
0 \ar[r] &
H^0(G,\cO_G) \ar[r] &
H^0(G,\cO_G(\mu A)) \ar[r] &
H^0(G,\cO_{\mu D}(\mu A)) \ar[r]^(.55){\alpha} &
H^1(G,\cO_G) \ar[r] &
0
} \]
%
Since both $\alpha$ and~$\beta_1$ are surjective,
we have that $\beta_2$ is surjective.
Now $\chi(\cO_G)=0$ so $h^1(G,\cO_G)=1$ and we conclude
that $H^1(X,\cF(\mu A))\not=0$.

From a local calculation at the geometric points of~$D$ we have
\[
h^0(X,\cG) \le
 3 \begin{pmatrix}\mu+1\\2\end{pmatrix}
\]
and by Riemann--Roch
\[
h^0(X,\cO_X(\mu A))= 3\mu(\mu+1)/2 + 1 \ge h^0(X,\cG)+1.
\]
Thus $h^0(X,\cF(\mu A))\ge 2$.

The linear system $\cH$ is the system associated to $H^0(X,\cF(\mu A))$,
and so it has positive dimension; $\cH_Y$ has the same dimension.
Since $\mu G\in\cH$,
the only possible fixed curve of~$\cH$ is some multiple~$\mu' G$,
but then $(\mu-\mu')G$ contradicts the minimality of~$\mu$;
therefore $\cH$ is mobile.
Let $H_Y\in \cH_Y$ be a general element. Since $H_Y\sim
\mu\pi^*(A) - \mu E$, and~$E^2=3$, we have $H_Y^2=0$.
So the map $\phi_Y=\phi_{\cH_Y}$
is a morphism to a curve. Furthermore, $H_Y\sim -K_Y$ so the general
fibre is a nonsingular curve (over~$k$) with trivial canonical class.
Since $\mu \pi_*^{-1}(G)$ is a fibre of~$\phi_Y$, the image
curve~$B$ has a rational point~$Q\in B$.  The minimality of~$\mu$
implies that $\cH_Y$ is the pencil $\phi_Y^*|\cO_{B}(Q)|$.

In cases A2 and~A3, we make similar calculations on a blowup of~$X$.
For example, in case~A2 let $\tau \colon X' \rightarrow X$ be the
blowup of~$P_2$ with exceptional curve~$L$.
Define $G'$ and~$\cH'$ to be the birational
transforms on~$X'$ of $G$ and~$\cH$ respectively.
The point $P_3$ lies on~$X'$, and we identify $P_1$ with its
preimage under~$\tau$.
Let $A' = \tau^*A - L$, and let $D' = P_1 + 2P_3$ as a divisor on~$G'$.

Define $\cF'$ as the sheaf $\cI_{P_1}^{\mu}\otimes \cI_{P_3}^{\mu}$ on~$X'$.
There is a map between exact sequences of sheaves of $\cO_{X'}$-modules
analagous to~\eqref{eq!Gdef} above (involving $A'$, $G'$, etc.)\
with $\cG'=\left(\cO_{X'}/\cF'\right)\otimes\cO_{X'}(\mu A')$.
Since $A'\sim -K_{X'}$, the argument works as before in cohomology,
with the conclusion that $H^1(X',\cF'(\mu A'))\not=0$.
The dimension calculation differs slightly, giving instead that
\[
h^0(X',\cG') \le
 2 \begin{pmatrix}\mu+1\\2\end{pmatrix}
\]
and
$h^0(X',\cO_{X'}(\mu A'))=
2\mu(\mu+1)/2 + 1 \ge
h^0(X',\cG')+1$.
The conclusion is again that $h^0(X',\cF'(\mu A'))\ge 2$,
and the rest of the proof follows verbatim.
In case~A3, the only change is again the dimension calculation.
\end{proof}

\section{Proof of the main theorem}
\label{sec!proof}

Let $\phi \colon X \broken B$ be as in the statement of
Theorem~\ref{thm!main}.

\subsection{Rationality of the base}
\label{sec!B}

Let $H_B$ be a very ample divisor on~$B$. We may choose it to
have minimal possible degree; since $B$ has genus~$0$, this is
either $1$ or~$2$. We first show that in fact the minimal degree
is always~1, so that~$B\cong\PP^1$.

Suppose $\deg H_B=2$; in particular, this means that $B$ has no
rational points.  We let $\cH=\phi^*|H_B|$. A general element
$H\in\cH$ splits over~$\kbar$ as a sum $D_1+D_2$ of two conjugate
curves each of genus~$1$.  Over $\kbar$, $D_1\sim D_2$, so the
class of~$D_1$ in $\Pic(\overline{X})$ is Galois invariant.  In
particular, $D_1$ defines a divisor class in~$\Pic(X)$ over~$k$.
So $H$ is divisible by~$2$ in~$\Pic(X)$: say $H\sim 2F$ where $F$
is an effective divisor defined over~$k$.
So $F\sim D_1$ over~$\kbar$ and therefore, over~$k$,
$|F|$ determines a map $X\broken\PP^1$ which factorises~$\phi$.
So $B$ has a rational point, contrary to our assumption.

\subsection{More preliminaries}
\label{sec!prelim}

We know now that $B$ has a rational point, so we may assume $B=\PP^1$.
We denote by~$\cH$ the mobile linear system $\phi^*|\cO_{\PP^1}(1)|$,
a linear system that defines~$\phi$.
Since $X$ is minimal, $\cH \subset |\mu A|$ for some fixed~$\mu\in\N$.
The anticanonical degree $\mu$ is also denoted~$\deg\cH$.

Let $P_1,\ldots,P_r$ be the distinct basepoints
of~$\cH$ and $m_1,\ldots,m_r \in \N$ their multiplicities: so a
general $C \in \cH$ has $\mult_{P_i}(C) = m_i$ for all~$i$.
The list $P_1,\ldots,P_r$ may include infinitely near basepoints
that lie on surfaces dominating~$X$; compare with Definition~\ref{def!res}.
Note that any $P_i$ may have degree greater than~1.

Let $f \colon W \to X$ be the blowup (in any appropriate order)
of all the~$P_i$; $f$~is a minimal resolution of indeterminacy
for~$\phi$.  We denote by~$E_i$ the \emph{total transform\/}
on~$W$ of the exceptional curve over~$P_i$: that is, if $L$ is the
exceptional curve of the blowup of~$P_i$ then $E_i$ is the total
transform of~$L$ on~$W$.
(Note that $E_i$ may be reducible or even nonreduced.)
Then denoting $\deg P_i$ by~$d_i$, we have
\begin{equation} \label{eqns!EiEj}
E_i^2 = -d_i \mbox{\quad and \quad} E_i
E_j = 0 \mbox{\;\; for $i \ne j$.}
\end{equation}
With this notation, the adjunction formula for~$f$ reads
\begin{equation}
\label{eq!KW}
K_W \sim f^*K_X + E_1 + \cdots + E_r
\end{equation}
and the birational transform $\cH_W$ of~$\cH$ on~$W$ satisfies
\begin{equation}
\label{eq!HW}
\cH_W \sim f^*\cH - m_1E_1 - \cdots - m_rE_r.
\end{equation}

\begin{thm}[Noether--Fano--Iskovskikh inequalities]
\label{thm!nfi}
Under the hypotheses of Theorem~\ref{thm!main}, $\cH$ has a basepoint of
multiplicity at least~$\mu = \deg \cH$: that is, $m_i \ge \mu$ for
some~$i$.
\end{thm}

\begin{rk}
\label{rk!Pi}
We may assume the
point~$P_i$ with $m_i \ge \mu$ is a point of~$X$, not an
infinitely near point, because multiplicities of linear systems
on nonsingular surfaces are nonincreasing under blowup.
\end{rk}

The theorem contrasts with the familiar case, explained in~\cite{CPR}
and \cite{KSC} \S5.1, for instance, when $\cH \subset |\mu A|$ induces
a birational map from $X$ to a nonsingular surface~$Y$ that is
minimal over~$k$: in this case the NFI inequalities tell us there
is a basepoint of multiplicity strictly larger than~$\mu$.  In
Mori theory the latter statement is that $\XmuH$ has a
noncanonical singularity; the case we need,
Theorem~\ref{thm!nfi}, says that $\XmuH$ has a nonterminal singularity.
For the modern viewpoint on NFI for elliptic and K3 fibrations
birational to Fano varieties, see~\cite{R06}, whose approach
follows Cheltsov~\cite{Ch} and is based on ideas of
Shokurov~\cite{Sh}.

\begin{proof}[Proof of Theorem~\ref{thm!nfi}]
By equations \eqref{eq!KW} and \eqref{eq!HW}
\[ 0 \:\:\qeq\:\: f^*\big(\KXmuH \big) \:\:\qeq\:\: K_W +
\tefrac{1}{\mu}\cH_W - \tesum_{i=1}^r \big(1 -
\tefrac{m_i}{\mu}\big) E_i \]
where $\qeq$~denotes $\Q$-linear equivalence of $\Q$-divisors.
Now the intersection number $\cH_W^2$ is zero since the morphism
$\phi \circ f$ is a fibration, which implies that
\begin{equation}
\tesum_{i=1}^r d_i m_i^2 = 3 \mu^2 \label{eqn!n1}.
\end{equation}
Also $K_W \cH_W = 0$ by the
adjunction formula, and expanding $\cH_W(K_W+(1/\mu)\cH_W) = 0$ gives
\begin{equation}
\tesum_{i=1}^r d_i m_i \big(1 - m_i/\mu \big) = 0. \label{eqn!n2}
\end{equation}
Now (\ref{eqn!n2}) implies the result, since if any of
the coefficients $(1 - m_i/\mu)$ is nonzero then at
least one must be negative.  Note that by
equation~(\ref{eqn!n1}) there is at least one basepoint, that is, $r
\ge 1$; this equation will also be used later.
\end{proof}

\subsection{Proof of Theorem~\ref{thm!main}.}
First we describe the logical structure of the argument.
It falls into two parts according to equation \eqref{eqn!n2}:
either $m_i>\mu$ for some~$i$, in which case we sketch a standard
induction step; or $m_i=\mu$ for every~$i$, and
we work this base case out in detail.

\paragraph{Induction step.}
This is essentially the proof of the birational rigidity
of~$X$, as given in~\cite{CPR}, for example.
We are given a point $P_i\in X$ (by Remark~\ref{rk!Pi}) with
multiplicity $m_i>\mu$ --- by definition, $P_i$ is a
{\em maximal centre} of~$\cH$. So
\[ 3\mu^2 \: = \: (\mu A)^2 \: = \: \cH^2 \: \ge \: m_i^2 d_i \:
> \: \mu^2 d_i, \]
where $d_i = \deg P_i$, and the inequality $\cH^2 \ge m_i^2 d_i$
is the global-to-local comparison of intersection numbers $\cH^2
\ge (\cH)^2_{P_i}$.  It follows that $d_i=1$ or~$2$.

We precompose~$\phi$ with the Geiser or Bertini involution~$i_{P_i}$.
It can be shown --- Lemma~2.9.3 of \cite{CPR} --- that this
\emph{untwists}~$\cH$, in other words that $\deg
(i_{P_i})^{-1}_*\cH < \deg \cH = \mu$, and we conclude by
induction on the degree~$\mu$.  (Note that if $\mu=1$ then all
$m_i=1$ by~\eqref{eqn!n2}.)

\paragraph{Base case.}
Equation~(\ref{eqn!n1}) implies that $\sum d_i = 3$, i.e., if
we count over an algebraic closure~$\kbar$ of~$k$
then there are 3~basepoints; we
must show they arise from Halphen data~$(G,D)$.

So let $\psi = f\circ \phi \colon W \to \PP^1$ be the morphism obtained
by blowing up the base locus $P_1,\ldots,P_r$ of~$\phi$.
We work over~$\kbar$ for the remainder of this paragraph.
Take a general fibre $F$ of~$\psi$; by Bertini's Theorem
$F$~is a nonsingular curve of genus~1.  Now
\[ F \;\sim\; \mu f^*(A) - \mu \tesum_{i=1}^r E_i
\;\sim\; {-\mu K_W}. \]
By Kodaira's canonical bundle formula applied
to~$\psi$,
\[ K_W \;\sim\; \psi^*(K_{\PP^1} + M) \,+\, \tesum_j
(n_j-1) G_j \]
where $M$ is a divisor of degree~$\chi(\cO_W)$
on~$\PP^1$ and the $n_j G_j \sim F$, with $n_j \ge 2$,
are the multiple fibres of~$\psi$.  Now
$\chi(\cO_W) = \chi(\cO_X) = 1$ so $M$ is a point and we have
\[ {-\tefrac{1}{\mu} F} \;\qeq\; {-F} + \tesum_j
\big(1-\tefrac{1}{n_j} \big) F. \]
Therefore $1-\frac{1}{\mu} = \sum_j
\big(1-\frac{1}{n_j} \big)$.  So either $\mu = 1$ and
there are no multiple fibres, or there is a single
multiple fibre $n_1 G_1 = \mu G_1 \sim F$ of
multiplicity~$\mu$. Since the subscheme of multiple fibres is
Galois invariant, $G_1$ is in fact defined over~$k$.
From here on, we work exclusively over~$k$.

In the case $\mu=1$, $\cH$ is a
pencil contained in~$\left|A \right|$ so it gives a linear
fibration and we are done.  The main case is~$\mu>1$.
Let $G_W = G_1$ and $G = f_*(G_W)$: then
\[ G \;\qeq\; f_* \big( \tefrac{1}{\mu} F \big) \;\qeq\;
f_*({-K_W}) \;\sim\; {-K_X} \;=\; A \]
so $G$ is a plane section of~$X$.  By minimality of~$X$,
$G$~is irreducible over~$k$; also
$\mu G = f_*(\mu
G_W) \in f_*(\cH_W) = \cH$, so $\mult_{P_i}(G) \ge 1$
for each basepoint~$P_i$.  (We are abusing notation
here: if $P_i$ is an infinitely near point, let $Z$
denote any surface between $X$  and~$W$ on which $P_i$~lies
and define $\mult_{P_i}(G)$ to be $\mult_{P_i}(G_Z)$, where
$G_Z$ is the pushforward of~$G_W$ to~$Z$.)
We claim that in fact
$\mult_{P_i}(G) = 1$ for each~$P_i$.  Indeed, first note
that $G_W$ is the strict transform of~$G$ on~$W$, since
otherwise $G_W$ would contain some~$E_i$ with multiplicity
at least~1; but then $E_i$ would be contained in a fibre
of~$\psi$, contradicting
\[ FE_i \;=\; {-\mu K_W E_i} \;=\; \mu d_i \;>\; 0. \]
Therefore the claim $\mult_{P_i}(G) = 1$ for each~$P_i$ is
equivalent to
\[ G_W \;=\; f^*(G) - \tesum_i E_i; \]
but the latter follows from the facts $\mu G \in
\cH$, $\mu G_W \in \cH_W$ and $\cH_W = f^*(\cH) - \sum \mu
E_i$.

We now construct an effective $k$-rational divisor~$D$
of degree~3 on~$G$ by the inverse of the procedure in
Definition~\ref{def!res}. 
We define $D$ to be $\sum \ell_i P_i$ as a divisor on~$G$, where the
sum extends over basepoints~$P_i$ that lie on~$X$ (rather
than on a surface dominating~$X$)
and $\ell_i$ is some factor 1, 2 or~3 that we specify.
If the $P_i$ are all points of~$X$ then we set all $\ell_i=1$, so
$D = P_1+P_2+P_3$ (this is one of cases A1, B and~C).
If $P_1, P_2 \in X$ and $P_3$ lies above~$P_2$,
possibly after renumbering, then we set $\ell_1=1$ and~$\ell_2=2$,
so~$D = P_1+ 2P_2$ (case~A2).
Notice that in this case $P_3$ must be the unique intersection point
of the exceptional curve above~$P_2$ and the birational transform of~$G$,
so this procedure is indeed the inverse of the construction
in Definition~\ref{def!res}.
If $P_1 \in X$, $P_2$ lies over~$P_1$
and $P_3$ lies over~$P_2$,
then we set $\ell_1=3$, so $D = 3P_1$ (case~A3); again the points $P_i$ lie
on the strict transform of~$G$ at every stage.

Next we check that $(G,D)$ is Halphen data: the outstanding
point is that $\cO_G(H) \cong \cO_G(\mu D)$
for a general curve $H \in \cH$,
that is, that $H$ cuts out exactly $\mu D$ on~$G$.
At a point~$P$, the divisor of~$H$ on~$G$ is $i_P(H,G)P$,
where $i_P(H,G)$ denotes the local intersection number of $H$ and~$G$.
So we must show that for basepoints $P_i$ that lie on~$X$, we have
$i_{P_i}(H,G)=\ell_i\mult_{P_i}(H)$ for the $\ell_i$ defined above.
In cases A1, B and~C,  $H$ can be chosen so that at any basepoint~$P_i$
none of its branches is tangent to $G$ at~$P_i$ --- otherwise there
would be an additional infinitely near basepoint above~$P_i$ --- so
$i_{P_i}(H,G)=\mult_{P_i}(H)$ and all $\ell_i=1$ as required.
In case A2, using the notation above with $P_3$ the infinitely near point,
again $i_{P_1}(H,G)=\mult_{P_1}(H)$. So
\[
i_{P_2}(G,H) \:\:=\:\: GH - i_{P_1}(G,H) \:\:=\:\: 3\mu - \mu \:\:=\:\:
2\mu \:\:=\:\: 2\mult_{P_2}(H)
\]
and $\ell_2=2$ as required. Case~A3 is similar.

Finally, let $\mu'$ be the index of~$(G,D)$; $\mu'$ is a divisor of~$\mu$.
The construction of Theorem~\ref{thm!constr} now applies to~$(G,D)$
to give a pencil $\cP$ on~$X$ containing~$\mu' G$.
On~$W$, the multiple $(\mu/\mu')\pi_*^{-1}\cP$ is contained
in~$\cH_W$; since $\cH_W$ is a pencil, we have $\mu'=\mu$
and~$\cH=\cP$.

\section{Algorithms}
\label{sec!alg}

We describe algorithms to carry out our analysis of elliptic
fibrations; we assume without comment standard routines of computer algebra
such as Taylor series expansions, ideal quotients and primary
decomposition.
We also need the field~$k$ to be computable;
that is, we must be able to make standard computations in linear
algebra over~$k$ and work with polynomials, rational functions
and power series over~$k$ and in small finite extensions of~$k$.
The routines are expressed here in a modular way; we have implemented them
in the computer algebra system \magma\ \cite{Ma} closely
following this recipe.  Our descriptions below are self-contained
and we include them to support the code.

The initial setup of the cubic surface is this:
$R=k[x,y,z,t]$ is the homogeneous coordinate ring of~$\P^3$
and $R(X) = R/F=\oplus_{n\in\N}H^0(X,\cO(n))$
is the homogeneous coordinate ring of~$X$;
here $F=F(x,y,z,t)$ is the defining equation of~$X$,
a homogeneous polynomial of degree~$3$.

\paragraph{Overview of the computer code.}

The code can be used
to build examples of Halphen fibrations, as in Section~\ref{sec!hal},
and Geiser and Bertini involutions in order to twist Halphen
fibrations, as in Section~\ref{sec!GB}; using these in
conjunction, one can realise Theorem~\ref{thm!main} for
particular examples.  The central point in all of these is to
impose conditions on linear systems on~$X$.
We describe an algorithm to do this in Section~\ref{sec!imp};
this follows our code very closely.
Then we explain the applications in Section~\ref{sec!app}.

Finally we give an implementation of Theorem~\ref{thm!main}
in Section~\ref{sec!mainalg}. This requires two additional elements:
we need to compute the multiplicity of a linear system (not just
a single curve) at a point $P\in X$ and to analyse the base locus
of a linear system on~$X$.

\subsection{Imposing conditions on linear systems}
\label{sec!imp}
This is the central algorithm: given a (nonsingular,
rational) point $P\in X$ and positive integers
$d$ and~$m$, return the space of forms of degree $d$
on~$\P^3$ that vanish to order $m$ at~$P$ when
regarded as functions on~$X$ in a neighbourhood of~$P$.

\paragraph{Step 1: A good patch on the blowup of $X$ at $P$.}
Change coordinates so that
$P=(0:0:0:1)\in X\subset\P^3$ and so that the projective
tangent space $T_pX$ to~$X$ at~$P$ is the hyperplane $y=0$. Then
consider the blowup patch $(xz,yz,z)$ in local coordinates on~$X$ at~$P$.
Altogether, this determines a map $f\colon\A^3\rightarrow\P^3$
with exceptional divisor $\Eamb=(z=0)$.
The birational transform $\Xtil$ satisfies
$f^*(X) = \Xtil + \Eamb$ and the exceptional
curve of $f_{|\Xtil}\colon \Xtil\rightarrow X$ is $E=\Eamb\cap\Xtil$,
which is the $x$-axis in~$\Eamb$.

\paragraph{Step 2: Parametrise $\Xtil$ near the generic point of $E$.}
The local equation of~$\Xtil$ is $g=f^*(F)/z$.
The exceptional curve $E$ is the $x$-axis.  Working over
$K=k(x)$, $\widetilde{X}$ is the curve $g(y,z)=0$ in~$\A^2_K$,
and this is nonsingular at the origin (the generic point of~$E$).
Cast $g$ into the ring $k(x)[\![z]\!][y]$ and compute a root $Y$
of~$g$ as a polynomial in~$y$ --- this is the implicit function
$Y=y(z)\in K[\![z]\!]$ implied by $g(y,z)=0$ (with coefficients
in~$K$).

\paragraph{Step 3: Pull a general form of degree $d$ back along
the blowup.}  Let $N$~be the binomial coefficient $d+3$ choose
$3$ and let $p = a_1x^d + a_2x^{d-1}y + \cdots + a_Nt^d$ be a
form of degree~$d$ with indeterminate coefficients
$a_1,\dots,a_N$.  Compute $q(x,y,z) = f^*(p)$.

\paragraph{Step 4: Impose vanishing conditions on $q$.}
Evaluate $q$ at $y=Y$.
The result is a power series in~$z$ with coefficients in~$k(x)$
and the indeterminates $a_1,\dots,a_N$.
The condition that $p$ vanishes to order
at least $m$ at~$P\in X$ is just that the coefficient of
$z^i$ vanishes identically for $i=0,\dots,m-1$.
Each such coefficient is of the form $p_i(x,a_1,\dots,a_N)/q_i(x)$,
where $q_i(x)$ is a polynomial in~$x$ and $p_i$ is polynomial in~$x$
but linear in $a_1,\dots,a_N$.
Writing $p_i=\sum_j \ell_{i,j}(a_1,\dots,a_N)x^j$,
the coefficient of~$z^i$ is zero if and only if
$ \ell_{i,j}(a_1,\dots,a_N)=0$ for each~$j$.
This is finitely many $k$-linear conditions on the~$a_i$.

\paragraph{Step 5: Interpret the linear algebra on $X$.}
Choose a basis of the solution space~$U_0$ of the linear conditions
on $a_1,\dots,a_N$. This is almost the solution; if $d\ge 3$,
however, we must work modulo the equation $F$ of the surface~$X$.
This is trivial linear algebra: compute the span
$W_d=F\cdot\cO(d-3)$ of~$F$ in degree $d$, intersect with the
given solutions $W=W_d\cap U_0$, and then compute a complement
$U$ inside $U_0$ so that $U_0=W\oplus U$.  A basis of~$U$ gives
the coefficients (in the ordered basis of monomials of degree
$d$) of a basis of the required linear subsystem of
$|\cO_{\P^3}(d)|$.

\paragraph{Variation 1: working inside a given linear system.}

Rather than working with all monomials of degree~$d$, we can
start with a subspace $V \subset H^0(X,\cO(d))$ and impose
conditions on that.
We simply work with a basis of~$V$ throughout the calculation
in place of the basis of monomials used above.

\paragraph{Variation 2: non-rational basepoints.}

In our applications, the only nonrational basepoints~$P$ that we
need to consider have degree 2 or~3.
In the former case we can make a degree~2 extension
$k\subset k_2$ so that $P$ is rational after base change to~$k_2$.
Computing as before at one of the two geometric points of~$P$
gives $k_2$-linear conditions on the coefficients~$a_i$.
Picking a basis for $k_2$ over~$k$, we can split these conditions
into `real and imaginary' parts, and impose them all as
linear conditions over~$k$.
A similar trick works for points of degree~3.

\subsection{Applications of the central algorithm}
\label{sec!app}

\paragraph{Building Halphen fibrations from Halphen data.}

We are given Halphen data $(G,D)$ of index $\mu$ on~$X$, as in
Definition~\ref{def!hal}, and we need to construct the associated
Halphen system $\cH\subset |\mu A|$ of Definition~\ref{def!H}
by imposing conditions on~$|\mu A|$.

Recall the points $P_i$ that are blown up in
Definition~\ref{def!res} to make the resolution of~$(G,D)$.
In cases A1, B and C, we simply impose the basepoints of~$X$
as multiplicity $\mu$ basepoints of~$\cH$, using Variation~2
of the algorithm to handle nonrational basepoints.
In case A2, we need to impose the conditions at~$P_1$
and $P_3$ only --- for the latter we must blow up $X$ at
$P_2$ and compute on that new surface.
Similarly in case A3 we make two blowups and impose conditions
only at~$P_3$.

\paragraph{Geiser and Bertini involutions.}

As usual, let $A=\cO_X(1)$.  The Geiser involution at~$P$ is
given by the linear system $\cL=|2A-3P|$, and the Bertini
involution at~$P$ is given by $\cL=|5A-6P|$. Bases of these
linear systems are computed by the algorithm of
Section~\ref{sec!mainalg}; we start by computing any basis, which
determines a map $j_P\colon X \broken \P^3$.

However, it is important to choose the right basis. There are two
problems that may occur with our initial choice: the image of
$j_P$ may not be~$X$; and, even if it is, $j_P$ could be the
involution we want composed with a linear automorphism of~$X$.
Our solution is to mimic the geometric definition
of~$i_P$ in Section~\ref{sec!GB}.
For both Geiser and Bertini involutions we find five
affine-independent points and compute their images under both
$i_P$ and~$j_P$, and thus interpolate for the linear automorphism
$\tau$ of~$\P^3$ such that $i_P = \tau \circ j_P$.

In the Geiser case, if $L$ is a general line
through $P$ then the two residual points of $X\cap L$ are swapped
by the involution.
Typically, residual points arising as $X\cap L$ become geometric
only after a degree~2 base change, and different lines need
different field extensions. This is a bit fiddly in computer
code, but is only linear algebra.  (There may be a better
solution using the projection of~$X$ away from $P$ to $\P^2$ and
working directly with the equation of~$X$ expressed as a
quadratic over the generic point of~$\P^2$.)

For the Bertini involution, in order to compute a single point
and its image under $i_P$ we first find the unique line $L$ though~$P$
and the point $R\in X$ such that $L\cap X = \{P,R\}$.
Let $\Pi\supset L$ be a general plane containing $L$; $E = X\cap \Pi$
is a nonsingular cubic curve.
We make the Weierstrass model of~$(E,R)$ --- that is, we embed $E$
in a new plane $\P^2$ with $R$ as a point of inflexion.
In that model, we take a general line through~$R$ and compute
the two other (possibly equal) intersection points $(Q_1,Q_2)$
of that line with~$E$.
Then $Q_2=-Q_1$ in the group law on~$E$ with $R$ as zero, and the
Bertini involution maps $Q_1$ to~$Q_2$.  Of course it may
happen that the points $Q_i$ are not $k$-rational; but in that
case, as for the Geiser involution, we simply make a degree~$2$
field extension to realise them and separate `real and imaginary'
parts later.

\paragraph{Calculating multiplicities of linear systems.}

Suppose $\cH$ is a linear system on~$X$ and $P\in X$ a point of
degree~1. To compute the multiplicity of~$\cH$ at~$P$
we run the first three steps of the algorithm of
Section~\ref{sec!mainalg} and the first evaluation of Step~4.
The result is a power series in the variable~$z$, and the
multiplicity of~$\cH$ at~$P$ is the order of that power series.

Whether this works in practice depends on what implementation
of power series is being used. If power series are expanded
lazily with precision extended as required then it works as
stated; if they are
computed to a fixed precision then the algorithm is best applied
to compute lower bounds on multiplicities.  Fortunately we use it
only to identify maximal centres, for which a lower bound is
exactly the requirement.

\subsection{The main theorem: untwisting elliptic fibrations}
\label{sec!mainalg}

We are given a cubic surface $X\subset\P^3$ together with
a rational map $\phi\colon X \broken \P^1$ defined by
two homogeneous polynomials $f,g$ of common degree~$d$.
Equivalently, we may regard $\phi$ as a linear system
$\cH=\left< f,g \right> \subset H^0(X,\cO(\mu))$.
In outline, the algorithm is simple; it terminates
by the proof of Theorem~\ref{thm!main}, the main point being
that Step~3 below cannot be repeated infinitely often.

\paragraph{Step 0: Trivial termination.}
If the degree~$\mu$ is equal to 1 then stop: the pencil
must be a linear elliptic fibration. Return the pencil
and its base locus (which is trivial to compute).

\paragraph{Step 1: Basepoints.}
Ideally we would compute precisely the base locus of~$\cH$
as a subscheme of~$X$ and work directly with that.
But to avoid computing in local rings, our algorithm in
Section~\ref{sec!base} below computes a finite set
of reduced zero-dimensional subschemes of~$X$ that
supports the base locus. (In short, it solves $f=g=0$ on~$X$
and then strips off one-dimensional primary components.)
We call these {\em potential basepoints} of~$\cH$.

As in Section~\ref{sec!proof}, the degree of a maximal centre
is at most 2, so we discard any potential basepoints of
higher degree. We refer to any of these as a
{\em potential centre} of~$\phi$.

\paragraph{Step 1a: Check termination.}
If there are no potential centres then stop: the linear system
must be an Halphen system, and moreover we must be in case C
of Definition~\ref{def!res} --- that is, there is a single
basepoint of degree~$3$.
Return the system and its base locus.

\paragraph{Step 2: Multiplicities.}
Compute the multiplicity of the linear system $\cH$ at each
potential centre $P$ in turn. (At points of degree~2 we make a
quadratic field extension and calculate at one of the two resulting
geometric points.)
If $P$ has multiplicity $m>\mu$ then go to Step~3.
It may happen that no such $P$ exists, in which case:

\paragraph{Step 2a: Termination.}
This is the base case of the proof of Theorem~\ref{thm!main}.
The linear system gives an Halphen fibration and its base locus
consists of all the potential centres of multiplicity~$m=\mu$.
Return the linear system and its base locus.

\paragraph{Step 3: Untwist.}
If the maximal centre $P$ has degree~1 then compute the Geiser
involution $i_P\colon X\broken X$ at that point.  If it has
degree~2, compute the Bertini involution $i_P\colon X\broken X$.
In either case, replace~$\phi$ by $\phi\circ i_P$
and repeat from Step~0.

\subsection{Analysing base loci on surfaces}
\label{sec!base}

It remains to provide an algorithm for Step~1 above.
We work in slightly more generality with an arbitrary
linear system $\cL$ on~$X$ corresponding to a
subspace $V\subset H^0(X,\cO(d))$. The base locus $B=\Bs\cL$
of~$\cL$ is contained in the subscheme $B'\subset X$ defined by the
ideal $I=\left< V \right>\subset R(X)$; the algorithm below
returns the reduced set of associated primes of height~$\ge 2$
of~$B'$.

\paragraph{Step 0: Setup.}
$\cL$ is defined by a basis of~$V$, a finite set of homogeneous
polynomials $p_1,\dots,p_k$ of degree~$d$.
Let $I=\left< p_1,\dots,p_k,F \right>\subset R$; this is the
ideal of~$B'$ considered as a subscheme of~$\P^3$.

\paragraph{Step 1: Identify and remove codimension 1 components.}
Let $I_{\red}$ be the radical of~$I$ and let $P_1,\dots,P_N$ be
the height~1 associated primes of~$I_{\red}$. Let $J_0=I$ and,
for $i=1,\dots,N$, let $J_i=(J_{i-1} : P_i^{n_i})$ where
$n_i\in\N$ is minimal such that $J_i$ is not contained in~$P_i$.
This removes the codimension~1 base locus without removing
any embedded primes there (at least set-theoretically): the radical
of~$J_N$ is the ideal of the set of all isolated or embedded
basepoints.

\paragraph{Step 2: End.}
Let $K=\rad(J_N)$, the ideal of a reduced zero-dimensional scheme.
Let $R_1,\dots,R_M$ be the associated primes of~$K$.
Return this set of primes.

\section{Examples}
\label{sec!egs}

We have implemented computer code in the \magma\ computational
algebra system; together with instructions, it can be downloaded
at~\cite{BR}. We present some examples below to illustrate our code.
Here we work in~$\PP^3$ defined over $k=\Q$, which we input as:
{\small
\begin{verbatim}
> k := Rationals();
> P3<x,y,z,t> := ProjectiveSpace(k,3);
\end{verbatim}
}\noindent
The symbol {\small \tt >} is the \magma\ prompt.
In some cases below the output has been edited mildly.

\subsection{An Halphen fibration with $\mu=2$}

We start with the surface $X\colon(t^3 - x^3 + y^2z + 2xz^2 -
z^3=0)\subset\P^3$.
{\small
\begin{verbatim}
> X := Scheme(P3,t^3 - x^3 + y^2*z + 2*x*z^2 - z^3);
> IsNonsingular(X);
true
\end{verbatim}
}\noindent
The surface $X$ is not minimal --- for example, $z=x-t=0$ is a
line --- but we can still construct interesting elliptic
fibrations on it.  The $t=0$ section of~$X$ is an elliptic
curve~$G$ with origin $O=(0:1:0:0)$ and an obvious rational
$2$-torsion point $R=(1:0:1:0)$. (Of course, to construct the
example we started with this curve and extended to~$X$.)
{\small
\begin{verbatim}
> O := X ! [0,1,0,0];
> R := X ! [1,0,1,0];
\end{verbatim}
}\noindent
To make Halphen data with~$\mu=2$, we need
an effective, $k$-rational divisor $D$ on~$G$ of degree~3
for which $D-3O$ is $2$-torsion in~$\Pic(G)$.
We construct such~$D$ as follows. Let $L\subset \PP^3$ be the
line $y=t=0$ and define a point of degree~$2$ on~$X$ by
$L\cap X = \{R,P\}$: so $P$ is the union of the two
points $(\alpha:0:1:0)$ with
$\alpha^2 + \alpha -1=0$.
Define $D=P+O$ as a divisor on~$G$.
The pair $(G,D)$ is Halphen data of index~$\mu=2$.
In fact the construction of the Halphen system is in terms
of linear systems and points on~$X$, rather than on~$G$, so for
the calculation it only remains to construct~$P$.
{\small
\begin{verbatim}
> L := Scheme(P3,[y,t]);
> PandR := Intersection(X,L);
> P := [ Z : Z in IrreducibleComponents(PandR) | Degree(Z) eq 2 ][1];  P;
Scheme over Rational Field defined by x^2 + x*z - z^2, y, t
\end{verbatim}
}\noindent
We build the Halphen system by imposing
$D$ as base locus of multiplicity~2 on the linear
system $|2A|$, where $A$ is a hyperplane section of~$X$.
{\small
\begin{verbatim}
> A2 := LinearSystem(P3,2);
> H0 := ImposeBasepoint(X,A2,P,2);
> H := ImposeBasepoint(X,H0,O,2);
> H;
Linear system on Projective Space of dimension 3
with 2 sections: x^2 + x*z - z^2, t^2
\end{verbatim}
}\noindent
The resulting fibration is $\phi=(x^2+xz-z^2 : t^2)\colon
X\broken\PP^1$, and we see $\phi^{-1}(1:0)=2G$.  We check that
the fibre $C = \phi^{-1}(-1:1)$ is irreducible and has genus~1:
{\small
\begin{verbatim}
> C := Curve(Intersection(X, Scheme(P3, t^2 + x^2 + x*z - z^2)));
> assert IsIrreducible(C);                                       
> Genus(C);                                                      
1
\end{verbatim}
}

\subsection{Geiser and Bertini involutions}

We construct a Geiser involution on the minimal surface
$X\colon (x^3 + y^3 + z^3 + 3t^3=0)\subset\PP^3$.
{\small
\begin{verbatim}
> X := Scheme(P3,x^3 + y^3 + z^3 + 3*t^3);
> P := X ! [1,1,1,-1];
> iP := GeiserInvolution(X,P);
> DefiningEquations(iP);
\end{verbatim}
}\noindent
returns the equations of the involution~$i_P$:
\begin{eqnarray*}
(   \; -xy + y^2 - xz + z^2 - 3xt - 3t^2 \;:\;
    x^2 - xy - yz + z^2 - 3yt - 3t^2 \;: \mbox{\hspace{20mm}} \\
  x^2 + y^2 - xz - yz - 3zt - 3t^2 \;:\;
    -x^2 - y^2 - z^2 - xt - yt - zt\; ).
\end{eqnarray*}
Since $P\in X$ is not an Eckardt point --- we discuss that case below ---
the Geiser involution contracts the tangent curve $C_P = T_P(X) \cap X$
to~$P$.
{\small
\begin{verbatim}
> TP := TangentSpace(X,P);
> CP := Curve(Intersection(X,TP));
> iP(CP);
Scheme over Rational Field defined by z + t, y + t, x + t
> Support(iP(CP));
{ (-1 : -1 : -1 : 1) }
\end{verbatim}
}\noindent
To make a Bertini involution, we find a point of degree~$2$.
{\small
\begin{verbatim}
> L := Scheme(P3,[x-y,z+t]);
> XL := Intersection(X,L);
> Q := [ Z : Z in IrreducibleComponents(XL) | Degree(Z) eq 2 ][1];
> iQ := BertiniInvolution(X,Q);
> DefiningEquations(iQ);
\end{verbatim}
}\noindent
again returns the equations of~$i_Q$, although in this case they are too
large to print reasonably: the first equation has 38 terms, beginning with
\[
   6x^2y^3 - 5xy^4 + 5y^5 - x^2y^2z - xy^3z - 4x^2yz^2 - 4y^3z^2 + 6x^2z^3 - 
        4xyz^3 + 11y^2z^3 - \cdots.
\]

\subsection{Eckardt points}
A $k$-rational point $P\in X$ is an {\em Eckardt point}
if $T_P X \cap X$ splits as three lines through~$P$
over a closure~$\overline{k}\supset k$.
For example, the surface
\[
X\colon (x^3 + y^3 + z^3 + 2t^3 = 0) \subset \P^3
\]
is minimal and $P=(1:-1:0:0)\in X$ is an Eckardt point:
$T_P X \cap X = (x+y=z^3 + 2t^3 = 0)$.
Geiser involutions in Eckardt points are in fact biregular,
and we see this here:
{\small
\begin{verbatim}
> X := Scheme(P3, x^3 + y^3 + z^3 + 2*t^3);
> P := X ! [1,-1,0,0];
> iP := GeiserInvolution(X,P);
\end{verbatim}
}\noindent
When \magma\ computes a map to projective space, it does not
automatically search for common factors between the defining
equations and cancel them.  To see the map more clearly, we
do this by hand.
{\small
\begin{verbatim}
> [ f div GCD(E) : f in E ] where E is DefiningEquations(iP);
[ y, x, z, t ]
\end{verbatim}
}\noindent
So the Geiser involution~$i_P$ switches $x$ and~$y$ in this
case, and that is clearly a biregular automorphism of~$X$.

\subsection{An example of untwisting}

Working on the same surface $X\colon (x^3 + y^3 + z^3 + 2t^3 =
0)$ as above, consider the fibration $f =(f_1:f_2)\colon X
\broken \PP^1$ defined by the two polynomials
{\small
\begin{center}
$f_1 =
57645x^2y^3 + 47234xy^4 - 9963y^5 + 23490x^2y^2z + 97322xy^3z +
70056y^4z - 26730x^2yz^2 - 33603xy^2z^2 + 5751y^3z^2 +
47925x^2z^3 + 85664xyz^3 - 5373y^2z^3 + 41480xz^4 + 72990yz^4 +
4095z^5 + 8100x^2y^2t + 157516xy^3t + 148392y^4t - 200880x^2yzt -
25896xy^2zt + 182664y^3zt + 9720x^2z^2t - 10800xyz^2t -
42408y^2z^2t + 118912xz^3t + 194220yz^3t + 109800z^4t -
124740x^2yt^2 - 27990xy^2t^2 + 96462y^3t^2 - 42120x^2zt^2 -
112938xyzt^2 - 70722y^2zt^2 + 24042xz^2t^2 + 28314yz^2t^2 +
63558z^3t^2 + 118530x^2t^3 + 111736xyt^3 - 48186y^2t^3 +
157684xzt^3 + 176616yzt^3 + 14958z^2t^3 + 247316xt^4 + 338796yt^4
+ 265536zt^4 + 123444t^5$
\end{center}
}\noindent
and
{\small
\begin{center}
$f_2=
20232x^2y^3 + 27216xy^4 + 6600y^5 - 66429x^2y^2z - 29187xy^3z +
40250y^4z + 25596x^2yz^2 - 8532xy^2z^2 - 42800y^3z^2 +
24507x^2z^3 + 23436xyz^3 + 3585y^2z^3 - 4185xz^4 + 35420yz^4 -
38240z^5 - 48978x^2y^2t + 77706xy^3t + 128092y^4t - 84456x^2yzt -
85428xy^2zt - 11724y^3zt + 65322x^2z^2t + 26676xyz^2t -
8214y^2z^2t + 100710xz^3t + 125152yz^3t + 25500z^4t -
196596x^2yt^2 - 75438xy^2t^2 + 122086y^3t^2 - 106596x^2zt^2 -
104598xyzt^2 + 366y^2zt^2 + 4590xz^2t^2 - 6786yz^2t^2 +
144574z^3t^2 - 62424x^2t^3 - 63612xyt^3 - 16932y^2t^3 -
105030xzt^3 + 1972yzt^3 - 98056z^2t^3 + 117720xt^4 + 231884yt^4 +
36888zt^4 + 247412t^5$.
\end{center}
}\noindent
Amazingly enough, this is an elliptic fibration --- although that
is by no means obvious, and we gave up on computing the genus of a
fibre with \magma\ after 5~hours.  To understand~$f$, we follow
the proof of Theorem~\ref{thm!main} as the algorithm of
Section~\ref{sec!mainalg}.  First we look for a maximal centre.
{\small
\begin{verbatim}
> P1 := ProjectiveSpace(k,1);
> f := map< P3 -> P1 | [f1,f2] >;
> time existence, Q := HasMaximalCentre(f,X); assert existence;
Time: 64.240
\end{verbatim}
}\noindent
This function, which executes Steps~1 and~2 of
Section~\ref{sec!mainalg}, returns either one or two values:
first, either true or false according to whether $f$ has a
maximal centre or not; and, second, a maximal centre if there is one.
In this example there is a maximal centre of degree~$2$:
{\small
\begin{verbatim}
> Q;
Scheme over Rational Field defined by
z^2 - 31/4*z*t - 5/4*t^2,  x + 3/2*z + 3/2*t,  y - 3/2*z - 1/2*t
> Degree(Q);
2
\end{verbatim}
}\noindent
We don't need to know it, but in fact $Q$ is the following pair
of conjugate points:
{\small
\begin{verbatim}
> k2<w> := Degree2SplittingField(Q);
> Support(Q,k2);
{ (w : -w - 1 : 1/3*(-2*w - 3) : 1),
      (1/8*(-8*w - 117) : 1/8*(8*w + 109) : 1/12*(8*w + 105) : 1) }
\end{verbatim}
}\noindent
Here $k_2$ is the number field $\Q[w]/(8w^2 + 117w + 135)$.

Following Step~3 of Section~\ref{sec!mainalg}, we untwist $f$ using
the Bertini involution $i_Q$ centred at~$Q$.
{\small
\begin{verbatim}
> iZ := BertiniInvolution(X,Z);
> g := iZ * f;
\end{verbatim}
}\noindent
As before, the defining equations of~$g$ have not been simplified by
\magma, and are of degree~$25$ with thousands of terms and
no common factor. However, a simple interpolation shows that $g$
is the map $(x:y)$.
We omit the demonstration of this here, but instead confirm it
by cross multiplication.
{\small
\begin{verbatim}
> Eg := DefiningEquations(g);
> assert IsDivisibleBy(x*Eg[2] - y*Eg[1], DefiningEquation(X));
\end{verbatim}
}\noindent

\subsection{The problem of minimality}

Geiser and Bertini involutions exist whether or not the
surface~$X$ is minimal: the geometric descriptions given in
Section~\ref{sec!GB} work regardless. In the
nonminimal case, however, the linear systems that determine the
involutions need not be $|2A-3P|$ and~$|5A-6P|$.  Here we give an
example where $|5A-6P|$ does not give a Bertini involution.

Let $X=(xt^2 + x^2y+y^3-z^3=0)\subset\P^3$.
The point $P = (0:0:0:1)$ is an Eckardt point
with tangent curve splitting as a line $x=y-z=0$ and
a conjugate pair of lines $x=y^2+yz+z^2=0$.
The point $Q=(1:0:0:0)$ lies on three conics,
each defined by $xy=t^2$ together with one of
the linear factors of~$y^3-z^3$.
Clearly each of the conics meets exactly
one of the lines, and that intersection is tangential.
The three intersection points
are $(0:1:1:0)$, $(0:\omega:1:0)$ and $(0:\omega^2:1:0)$
where $\omega$~is some chosen primitive cube root of~$1$.
Let $Z=(x=t=y^2+yz+z^2=0)\subset X$ be the conjugate pair
of intersection points.
Although $X$ is clearly not minimal, we can compute
the linear system~$|5A-6Z|$.
{\small
\begin{verbatim}
> X := Scheme(P3,x*t^2 + x^2*y + y^3 - z^3);
> Z := Scheme(P3,[x,t,y^2+y*z+z^2]);
> L1 := ImposeBasepoint(X, LinearSystem(P3,5), Z, 6);
> L2 := Complement(L1,X);
\end{verbatim}
}\noindent
Notice that since the linear system is computed on the ambient~$\PP^3$,
we must work modulo the equation of~$X$ by hand, taking a
complement of the subspace of degree~$6$ polynomials that it
divides --- in previous examples this was hidden inside the
function for Bertini involutions.

But this is the wrong linear system; it has (projective) dimension~4:
{\small
\begin{verbatim}
> #Sections(L2);
5
\end{verbatim}
}\noindent
Our code cannot compute the Bertini involution in this case.
Out of interest, we show instead how to make the map $f\colon X\broken\PP^4$
with these five sections and compute its image.
{\small
\begin{verbatim}
> P4<[a]> := ProjectiveSpace(k,4);
> f := map< P3 -> P4 | Sections(L2) >;
> f(X);
\end{verbatim}
}\noindent
returns a surface in~$\P^4$ defined by three equations,
the $2\times 2$ minors of the $2\times 3$ matrix
\[
\begin{pmatrix}
\; -a_4 \; & \; a_1^2 + a_2^2 + a_2a_3 + a_3^2 \; & \; a_1 \; \\
\; a_5 \; & \; a_4^2 - a_1a_3 \; & \; a_2-a_3 \;
\end{pmatrix}.
\]
The third minor is the equation of~$X$; the second is
the cone on $\P^1\times\P^1$ in some coordinates.
In fact, this image surface is singular: it has a single
Du Val singularity of type~${\mathrm A}_2$.
The map~$f$ blows up~$Z$
and then contracts the two conjugate lines that meet at~$P$,
which form a chain of two $-2$-curves on the blowup.

\vspace{5mm}
\noindent
Gavin Brown, IMSAS, University of Kent, CT2 7AF, UK.\\
{\small \texttt{gdb@kent.ac.uk} }

\vspace{5mm}
\noindent
Daniel Ryder, Dept.\ of Mathematics, University Walk,
Bristol, BS8 1TW, UK.\\
{\small \texttt{daniel.ryder@bristol.ac.uk} }

\end{document}